\documentclass[11pt]{amsart}
\usepackage[margin=1in]{geometry}
\usepackage{amsmath, amsfonts, amsthm, amssymb}
\newtheorem{thm}{Theorem}[section]

\newtheorem{lem}[thm]{Lemma}
\newtheorem{prop}[thm]{Proposition}
\theoremstyle{definition}

\theoremstyle{remark}
\newtheorem{rem}[thm]{Remark}
\newtheorem{exam}[thm]{Example}
\theoremstyle{conjecture}

\numberwithin{equation}{section}
\begin{document}

\title[Notes On Quadratic Integers and Real Quadratic Number Fields]
 {Notes On Quadratic Integers and Real Quadratic Number Fields}

\author{ Park, Jeongho }

\address{Department of Mathematics,
POSTECH,
San 31 Hyoja Dong, Nam-Gu, Pohang 790-784, KOREA.
Tel. 82-10-3047-7793.}

\email{pkskng@postech.ac.kr}
%11R29: Class numbers, class groups, discriminants
%11Y40: Algebraic number theory computations
%11J68: Approximation to algebraic numbers
%11R11: Quadratic extensions
\thanks{Supported by the National Research Foundation of Korea(NRF) grant funded by the Korea government(MEST)
(2010-0026473).}

\begin{abstract}
It is shown that when a real quadratic integer $\xi$ of fixed norm $\mu$ is considered, the fundamental unit $\varepsilon_d$ of the field $\mathbb{Q}(\xi) = \mathbb{Q}(\sqrt{d})$ satisfies $\log \varepsilon_d \gg (\log d)^2$ almost always. An easy construction of a more general set containing all the radicands $d$ of such fields is given via quadratic sequences, and the efficiency of this substitution is estimated explicitly. When $\mu = -1$, the construction gives all $d$'s for which the negative Pell's equation $X^2 - d Y^2 = -1$ (or more generally $X^2 - D Y^2 = -4$) is soluble. When $\mu$ is a prime, it gives all of the real quadratic fields in which the prime ideals lying over $\mu$ are principal.
\end{abstract}

% -----------------------------------------------------------
\maketitle
% -----------------------------------------------------------

\section*{Introduction}\label{sec_introduction}

The regulator is probably one of the most unpredictable constants related to a number field. Dirichlet's class number formula already gave a clear connection between the $L$-value, the class number, and the regulator, but while the first two have been admitting huge theories in various perspectives, the regulator seems to remain far from being exploited. Let $d$ be a positive square-free integer and $K$ the field $\mathbb{Q}(\sqrt{d})$ with discriminant $D$, class number $h_d$ and fundamental unit $\varepsilon_d$. In this case Dirichlet's formula reduces to
\begin{equation*}
h_d = \frac{\sqrt{D} L(1,\chi)}{2 \log{\varepsilon_d}}.
\end{equation*}
The $L$-value is known to stay in a relatively narrow range $D^{-\epsilon} \ll L(1,\chi) \ll \log D$ for arbitrary $\epsilon > 0$ \cite{Li}, so we know that $D^{1/2 - \epsilon} \ll h_d \log \varepsilon_d \ll D^{1/2 + \epsilon}$. As for the class number, many things can be said about the primary parts of ideal class groups. Nevertheless we know very little about the fundamental unit, despite the fact that the computation of $h_d$ is essentially impossible without computing $\log \varepsilon_d$.

Although $\log \varepsilon_d$ seems to vary in a wide range between $O(\log D)$ and $O(D^{1/2}\log D)$ in a somewhat uncontrolled way, concerning its average we have several precise conjectures in terms of $h_d$ \cite{Cohen_2, Hooley_1, Jacobson}. A problem that can be considered as a yardstick in this direction is to show that the regulator is in most cases much larger than the class number.

It is well known that the fundamental unit $\varepsilon_{d}$ comes from the continued fraction expansion of $\sqrt{d}$ or $\frac{\sqrt{d}+1}{2}$ whose period $l$ is the most dominant factor in the size of $\varepsilon_{d}$ \cite{Cohen_3}. Naturally, many researches have been focused on the size of period and now its upper bound $l \ll \sqrt{d} \log d$ has fairly precise versions \cite{Cohn,Podsypanin}. As for the lower bound, it is generally believed that $\log \varepsilon_d \gg D^{1/2 - \epsilon}$ holds for most of square-free integer $d$. However, the average order of $\varepsilon_d$ is known only to the following extent; for almost all non-square $d$, one has $\varepsilon_d > d^{\frac{7}{4} - \epsilon}$ \cite{Fouvry_2,Fouvry_1}.

Keeping this difficulty in mind, in this article we lay our interest on quadratic integers instead. When a quadratic integer $\xi$ has norm a rational prime $p$, $\xi$ generates a principal prime ideal, i.e., $p$ splits or ramifies into principal prime ideals in $\mathbb{Q}(\xi) / \mathbb{Q}$. A simple fact is that the more there are rational primes $p \ll \sqrt{d}$ that split into principal primes $P \overline{P}$ in $\mathbb{Q}(\sqrt{d})$, the bigger the fundamental unit $\varepsilon_d$ becomes \cite{Yamamoto}. Another fact is that (see section~\ref{sec_Quadratic integers of small norms}) if $p < \sqrt{D}/2$ and $\xi$, $\tilde{\xi}$ are the least elements in $P$, $\overline{P}$ among those numbers greater than 1, then $\xi \tilde{\xi} = p \varepsilon_d$. This suggests that a full knowledge about a single principal prime and its conjugate gives full information about $\varepsilon_d$. This already justifies a study about quadratic integers, and we show in this article that $\log \varepsilon_d \gg \log^2 d$ is true in most cases, where `most' shall be interpreted in an adequate fashion based on quadratic integers.
\bigskip

Now we formulate the main result. Let $\mu$ be a square-free integer with $|\mu|>1$ and $\xi$ a quadratic integer, i.e., a zero of a monic polynomial $X^2 - TX + \mu$ for some integer $T$. Since we will treat real quadratic fields which appear only when $T^2 - 4 \mu > 0$, we shall fix the norm $\mu$ of $\xi$ and let $|T| \rightarrow \infty$. For each integer $T$ with $T^2 - 4\mu > 0$, let $\xi_{\mu}(T)$ be the large root of $X^2 - TX + \mu=0$ and $D=D_{\mu}(T)$ the discriminant of $\mathbb{Q}(\xi_{\mu}(T))$ when $\xi_{\mu}(T)$ is irrational. Observe that $D_{\mu}(T)=D_{\mu}(-T)$, so it is no harm to consider positive $T$ only. Let $d = d_{\mu}(T)$ be the radicand corresponding to $D$, viz., $d = D$ if $D \equiv 1$ mod $4$ and $D/4$ otherwise. Let $f_{|\mu|}(N)$ be the number of distinct fields in $\{ \mathbb{Q}(\xi_{-\mu}(T))\;|\; 1<T<N\} \bigcup \{ \mathbb{Q}(\xi_{\mu}(T))\;|\; 1<T<N\}$. $\omega(\mu)$ denotes the number of distinct prime factors of $\mu$.

Our main result is following

\begin{thm}\label{thm_main}
\begin{equation}\label{eqn_thm_main_1}
    \liminf_{N \rightarrow \infty}{\frac{f_{|\mu|}(N)}{N}} \geq 2^{-\omega(\mu)} \quad \text{ and }
\end{equation}

\begin{equation}\label{eqn_thm_main_2}
    \lim_{N\rightarrow \infty} \left( \frac{\sharp \left\{ 1 < T < N\;| \; \log \varepsilon_d > L_{\mu}(T)  \right\}}{N} \right) = 1, \quad \text{where}
\end{equation}

\begin{equation*}
    L_{\mu}(T) =  \frac{1}{\log |\mu|} \left( \log{ \frac{\sqrt{D}}{2}  } \right)^2 - \left( 3 - \frac{2 \log 2}{\log |\mu|} \right) \log{\frac{\sqrt{D}}{2}} - 2 \log 2 - \frac{2|\mu|}{|\mu|-1}.
    \end{equation*}
\end{thm}
\bigskip

Yamamoto already gave the infinitude of number fields that satisfy $\log \varepsilon_d \gg \log^{n+1} D$ for $n = 2$ \cite{Yamamoto}, but he did not discuss how often such fields arise in nature. On the other hand, Reiter tried to make Yamamoto's bound effective \cite{Reiter}, and in the process of doing so the original leading term $\frac{2^n}{(n+1)!\log p_1 \cdots \log p_n} \left( \log \sqrt{D} \right)^{n+1}$ was weakened to $\frac{1}{2(n+1)!\log p_1 \cdots \log p_n} \left( \log \frac{\sqrt{D}}{2} \right)^{n+1}$. Theorem~\ref{thm_main} gives a density result about the appearance of such fields for the case $n=1$, and the lower bound for the regulator is also tighter than that of Reiter.

Recalling how little we know about $\varepsilon_d$ in general, theorem~\ref{thm_main} suggests an interesting set of radicands, namely $\frak{D} = \frak{D}_{\mu} = \{ d_{\mu}(T)\;|\;T > 0,\; T^2-4\mu >0 \}$. It would be desirable to list up the elements of $\frak{D}$ according to their size and examine the density of $\frak{D}$ in $\mathbb{Z}$, or to construct the set $\frak{D}$ explicitly. Unfortunately we were not successful in this direction. Instead, we try to consider a bigger set that contains $\frak{D}$ and whose construction is simple, explicit and has a sort of measurable efficiency. For this, as long as $\xi_{\mu}(T)$ and $\xi_{\mu}(T')$ generate distinct ideals in $\mathbb{Q}(\sqrt{d})$, we allow the radicand $d$ to be counted again in $\frak{D}$.

Let $I^{(0)}(\mu) = \{ (y,x)\in\mathbb{Z}^2\;|\; 0\leq x < y,\; gcd(x,y)=1,\; x^2\equiv \mu\text{ (mod $y$)} \}$, $I^{(1)}(\mu) = \{ (y,x)\in I^{(0)}(\mu)\;|\; y\text{ is odd} \}$ and
\begin{equation*}
\tilde{y} = \begin{cases}
\frac{y}{2}\;\; &\text{ if $y$ is even}\\
y\;\; &\text{ otherwise.}
\end{cases}
\end{equation*}

We use $\mathbb{N}$ to denote the set of nonnegative integers. Define
\begin{align*}
    &\frak{D}^{(0)}(\mu;y;t) = \{ \tilde{y}^2k^2 + 2 k (\tilde{y}/y) \sqrt{\mu+y^2 t} + t\;|\; k\in\mathbb{N}  \},\\
    &\frak{D}^{(1)}(\mu;y;t) = \{ 4 y^2 k^2 + 4 k \sqrt{4\mu+ y^2 t} + t\;|\; k\in\mathbb{N}  \}
\end{align*}
which are quadratic progressions, and let
\begin{equation*}
    \widehat{\frak{D}}^{(j)}(\mu;y;t) = \{ d \in \frak{D}^{(j)}(\mu;y;t)\;|\; d\text{ is square-free}  \}.
\end{equation*}
Then we have
\begin{prop}\label{prop_d_mu}
(1) There exist maps $\phi_{\mu}^{(j)}:I^{(j)}(\mu) \rightarrow \frak{D}$ for $j=0,1$ such that
\begin{equation*}
\frak{D} \subset \bigcup_{j=0,1} \bigcup_{(y,x)\in I^{(j)}(\mu)} \frak{D}^{(j)}(\mu;y;\phi_{\mu}^{(j)}(y,x)).
\end{equation*}
(2) For each $(j,y,x)$ there exists an arithmetic progression $\{T^{(j)}_{y,x}(d)\;|\; d \in \frak{D}^{(j)}(\mu;y;\phi_{\mu}^{(j)}(y,x)) \}$ satisfying $d_{\mu}(T^{(j)}_{y,x}(d)) = d$ for $d \in \frak{D} \cap \frak{D}^{(j)}(\mu;y;\phi_{\mu}^{(j)}(y,x))$, such that
\begin{equation}\label{eqn_d_T_zeta sum}
     \sum_{j=0,1} \sum_{(y,x) \in I^{(j)}(\mu)} \sum_{d \in \frak{D}^{(j)}(\mu;y;\phi_{\mu}^{(j)}(y,x)) } \frac{1}{\left(T^{(j)}_{y,x}(d)\right)^s}
\end{equation}
has a simple pole at $s=1$ with residue $R(\mu) \ll 1$.
\end{prop}
\bigskip

The proof of proposition~\ref{prop_d_mu} is simple but is constructive and shows that the maps $\phi_{\mu}^{(j)}$ and $T^{(j)}_{y,x}$ can be determined in a canonical way.

What proposition~\ref{prop_d_mu} says is the following. Since $|\xi_{\mu}(T) - T| \ll |\mu/T|$, one has $\sum \xi_{\mu}(T)^{-s} - \sum T^{-s} \ll 1$ as $s \rightarrow 1+$. It turns out by theorem~\ref{thm_E_mu_x} that even when we pick the traces $T$ in a way that the ideals generated by $\xi_{\mu}(T)$ are all distinct, the series $\sum T^{-s}$ still have the same residue with $\zeta(s)$ at $s=1$. We may say for this that the residue corresponding to $\frak{D}$ is 1. In proposition~\ref{prop_d_mu}, (1) suggests that $\frak{D}$ can be approximated by a union of quadratic progressions, and by (2), its correspondent residue $R(\mu)$ may be considered as a measure of this approximation quality.

To specify how much this union is bigger than $\frak{D}$, let $\widehat{R}(\mu)$ be the residue in \eqref{eqn_d_T_zeta sum} that we shall obtain when $\frak{D}^{(j)}(\mu;y;\phi_{\mu}^{(j)}(y,x))$ is replaced by $\widehat{\frak{D}}^{(j)}(\mu;y;\phi_{\mu}^{(j)}(y,x))$ in the innermost sum. We first observe that $\widehat{R}(\mu)=1$. This is because $T^{(j)}_{y,x}(d)^2 - 4\mu = y^2d$ under our construction, and if we assume $d$ is square-free, there is only one quadruple $(j,y,x,d)$ satisfying $T^{(j)}_{y,x}(d) = T$ for a fixed $T$. Since $R(\mu)$ is the sum of all the residues of the innermost sums in \eqref{eqn_d_T_zeta sum}, computing the density of square-free numbers in $\frak{D}^{(j)}(\mu;y;\phi_{\mu}^{(j)}(y,x))$ shall give an estimate of $R(\mu)$. Write $[N] = \{ 1,2,\cdots,N\}$ and put
\begin{equation*}
    f_{y,x}^{(j)}(N) = |\frak{D}^{(j)}(\mu;y;\phi_{\mu}^{(j)}(y,x)) \bigcap [N]|,\;\;
    \widehat{f}_{y,x}^{(j)}(N) = |\widehat{\frak{D}}^{(j)}(\mu;y;\phi_{\mu}^{(j)}(y,x)) \bigcap [N]|.
\end{equation*}
Then we have

\begin{prop}\label{prop_d_squarefree}
For each pair $(y,x) \in I^{(j)}(\mu)$,
\begin{equation*}
    \lim_{N \rightarrow \infty} \frac{\widehat{f}_{y,x}^{(j)}(N)}{f_{y,x}^{(j)}(N)} = \left( 1 - \frac{\omega_d(2)}{4} \right) \cdot \prod_{p|y}'\left( 1 - \frac{1}{p^2} \right) \cdot \prod_{p^2|\mu}'\left( 1 - \frac{1}{p} \right) \cdot \prod_{p \nmid \mu y,\;\left(\frac{\mu}{p}\right)=1}' \left( 1 - \frac{2}{p^2} \right)
\end{equation*}
where the restricted products are over odd primes, and
\begin{equation*}
    \omega_d(2) =
    \left\{
      \begin{array}{ll}
        2 & \hbox{if $j=0$, $y$ is odd and $\mu \equiv 0$ or $1$ (mod $4$);} \\
        2 & \hbox{if $j=0$, $y \equiv 2$ (mod $4$), $\mu \equiv 1$ (mod $8$);} \\
        1 & \hbox{if $j=0$, $y \equiv 0$ (mod $4$)} \\
        0 & \hbox{otherwise.}
      \end{array}
    \right.
\end{equation*}
\end{prop}
\bigskip

\begin{exam}
Let $C_{\mu}^{(j)}(y,x)$ be the limit in proposition~\ref{prop_d_squarefree} and consider $\mu=2$. For any pair $(y,x)\in I^{(0)}(2)$, $2$ must be a square modulo $y$. If $y$ is odd, this is possible exactly when every prime factor of $y$ is congruent to $\pm 1$ modulo 8. Since $gcd(x,y)=1$, $y$ cannot be even and we thus have $I^{(0)}(2) = I^{(1)}(2) = \{(1,0),(7,3),(7,4),(17,6),(17,11),\cdots \}$. It is easy to see that $\omega_d(2) = 0$ always, and numerically one can show that
\begin{align*}
    C_{2}^{(j)}(y,x) &= \prod_{p|y}\left( 1 - \frac{1}{p^2}\right) \cdot \prod_{\substack{p \nmid y\\p \equiv \pm 1(\text{mod } 8)}} \left( 1 - \frac{2}{p^2}\right)\\
    &\geq \prod_{p \equiv \pm 1(\text{mod } 8)} \left( 1 - \frac{2}{p^2}\right)\\
    &= 0.94\cdots
\end{align*}
which implies that $0.94 R(2) < \widehat{R}(2) = 1$. The union of quadratic progressions can be therefore considered as a nice substitution for $\frak{D}$ in this case. The choice $\mu = 2$ might seem to be especially good, but in fact $\prod_{p:odd} \left( 1 - \frac{2}{p^2}\right) > 0.64$ and the probability to get a square-free integer from $\frak{D}^{(j)}(\mu;y;\phi_{\mu}^{(j)}(y,x))$ is very often close to 1, and hence so is $R(\mu)$. Note that the product in proposition~\ref{prop_d_squarefree} does not involve any odd prime $p$ with $\left(\frac{\mu}{p}\right)=-1$, and the value of the infinite product is affected very little by large primes. The probability $C_{\mu}^{(j)}(y,x)$ is therefore particularly close to 1 when $\mu \equiv 2$ or $3$ (mod 4), $\mu$ is square-free and $\left(\frac{\mu}{p}\right)=-1$ for small odd primes $p$ so that $4 \nmid y$, $\omega_d(2) = 0$ and many small primes are excluded in the infinite product.
\end{exam}

Next remarks are all about proposition~\ref{prop_d_mu}.

\begin{rem}\emph{Quadratic units.}
We assumed $\mu$ is square-free, but proposition~\ref{prop_d_mu} and \ref{prop_d_squarefree} are valid for $\mu = \pm 1$ too and quadratic units can be dealt with in the same way. Suppose $\mu = \pm 1$. If $d \in \widehat{\frak{D}}^{(0)}(\mu;y,\phi_{\mu}^{(0)}(y,x))$, the fundamental solution to the Pell's equation $X^2 - d Y^2 = \mu$ is $\left( \lfloor \sqrt{d} \rfloor y + x \right)^2 - d y^2 = \mu$. Once $\phi_{\mu}^{(j)}$ is determined naturally, the notion of \emph{leasts to $i$} in ~\cite{JHPark2} means the least element of $\widehat{\frak{D}}^{(0)}(\mu;y,\phi_{\mu}^{(0)}(y,x))$ for each $(y,x) \in I^{(0)}(\mu)$. So proposition~\ref{prop_d_mu} and ~\ref{prop_d_squarefree} may be considered as a generalization of the result in \cite{JHPark2} to quadratic integers with norm other than $\pm 1$.
\end{rem}

\begin{rem}\emph{Ramification and fundamental units.}
Suppose $d \equiv 2,3$ (mod 4) and
\begin{equation*}
    d \in \frak{D}^{(0)}(p;y;\phi_{p}^{(0)}(y,x)),\;\; \sqrt{d}-1 > p\;\;\text{and $p$ is ramified in $\mathbb{Q}(\sqrt{d})/\mathbb{Q}$}.
\end{equation*}
Then (see section~\ref{sec_Quadratic integers of small norms} and the proof of proposition~\ref{prop_d_mu})
\begin{equation}\label{eqn_remark_ramifi_funda_unit}
    \varepsilon_d = \frac{1}{p} \left( \lfloor \sqrt{d} \rfloor y + x + y \sqrt{d} \right)^2.
\end{equation}
Suppose that a real quadratic number field with discriminant $D = 4d$, $\sqrt{d}-1 > 2$ has class number 1. Then 2 is ramified into a power of a principal ideal, say $(\xi)^2$, where $\xi$ comes from the continued fraction of $\sqrt{d}$ and corresponds to a pair $(y,x) \in I^{(0)}(2) \cup I^{(0)}(-2)$. It follows that all such fields have fundamental unit of the form \eqref{eqn_remark_ramifi_funda_unit} for some $(y,x) \in I^{(0)}(\mu)$ where we take $\mu=\pm 2$.
\end{rem}

\begin{rem}\emph{Decomposition into principal primes.}
In case $\mu$ is a prime, say $p$, the set $\frak{D}_{p} \cup \frak{D}_{-p}$ gives a complete list of real quadratic number fields in which $p$ ramifies or splits into principal ideals. This is because any principal prime ideal of norm $p$ has a generator $\xi$ of norm $\pm p$ and $\frak{D}_p \cup \frak{D}_{-p}$ contains all of the positive radicands of such fields. Since the prime ideal over $p$ is always principal when $p$ is inert, this classifies all radicands for which the prime ideals of $\mathbb{Q}(\sqrt{d})$ over $p$ are principal. A square-free integer $d$, though, can be contained in several quadratic progressions given by the construction.
\end{rem}

\bigskip

The article is organized as follows. Section ~\ref{sec_Preliminaries} gives requisites briefly. In section~\ref{sec_Reduced ideals} the relations between principal reduced ideals and reduced quadratic irrationals coming from the continued fraction of $\sqrt{d}$ or $\frac{1+\sqrt{d}}{2}$ are described. In section~\ref{sec_Quadratic integers of small norms} we state the distribution of minimal quadratic integers $\xi$, where a quadratic integer $\xi$ is \emph{minimal} if it is the least number greater than 1 in the ideal $(\xi)$ it generates.  Using the contents in sections from ~\ref{sec_Reduced ideals} to ~\ref{sec_Quadratic integers of small norms}, the proof of theorem~\ref{thm_main} is given in section~\ref{sec_proof of main theorem}. Section~\ref{sec_the quadratic progressions} covers details for proposition~\ref{prop_d_mu} and ~\ref{prop_d_squarefree}. In section~\ref{sec_further topics} we discuss some technical issues for further research.

%-----------------------------------------------------------------------------
\section{Preliminaries}\label{sec_Preliminaries}

In this article $\mu$ represents a square-free integer whose absolute value is comparatively small, and $d$ a square-free positive integer which is usually considered to be large. Let $K_d = \mathbb{Q}(\sqrt{d})$ and $O_d$ the ring of integers of $K_d$, $D$ the discriminant of $K_d$. Put
\begin{equation*}
\omega_d =\begin{cases}
\frac{1 + \sqrt{d}}{2} & \text{if $d \equiv 1$ mod $4$},\\
\sqrt{d} & \text{otherwise}.
\end{cases}
\end{equation*}

Let $\omega_d = [a_0,a_1,a_2,\cdots]$ be the simple continued fraction expansion of $\omega_d$, $l$ the period of $\omega_d$, $p_n/q_n = [a_0,a_1,\cdots,a_n]$ a convergent to $\omega_d$, $\alpha_{n+1} = [a_{n+1},a_{n+2},\cdots]$ the (n+1)-th total quotient. By convention we put $(p_{-1},q_{-1}) = (1,0)$, $(p_{-2},q_{-2}) = (0,1)$. For $x \in \mathbb{Q}(\sqrt{d})$, let $\overline{x}$ be its conjugate and $N(x) = x \overline{x}$. For the $n$-th convergent $p_n / q_n$ of $\omega_d$, put
\begin{equation*}
\xi_n = \overline{ p_n - q_n \omega_d} =\begin{cases}
p_n - q_n + q_n \omega_d & \text{if $d \equiv 1$ mod $4$}\\
p_n + q_n \omega_d & \text{otherwise}
\end{cases}
\end{equation*}

and let $\nu_n = |N(\xi_n)| = (-1)^{n+1}N(\xi_n)$. We say that a quadratic integer $\xi$ \emph{comes from} a convergent to $\omega_d$ when $\xi = \xi_n$ for some $n$.\\

Recall that a quadratic irrational $\alpha$ is \emph{reduced} if $\alpha > 1$ and $-1 < \overline{\alpha} < 0$. It is a classical result that the continued fraction expansion of a real number $x$ is purely periodic if and only if $x$ is a reduced quadratic irrational (for example, see theorem 7.2 of \cite{niven1991introduction}). In particular $\sqrt{d} + \lfloor \sqrt{d} \rfloor$ and $\frac{1 + \sqrt{d}}{2} + \lfloor \frac{1 + \sqrt{d}}{2} \rfloor - 1$ are reduced, so one can write $\omega_d = [a_0,\overline{a_1,\cdots,a_l}]$ where $a_l = 2 a_0 - 1$ if $d \equiv 1$ mod 4 and $a_l = 2 a_0$ otherwise. We also recall
\begin{prop}[\cite{Cohen_3}]\label{prop_xi_l_minus_1 and symmetry}
$\varepsilon_d = \xi_{l-1}$, and the sequence $\{ a_1,\cdots,a_{l-1}\}$ is symmetric.
\end{prop}

The following will be used freely. Let $x$ be a positive real number, $p_m/q_m$ be its $m$-th convergent and $\alpha_m$ the $m$-th total quotient.

\begin{prop}[\cite{Hardy}\label{prop_Hardy}]
If (p,q) = 1 and
\begin{equation*}
\left| \frac{p}{q} - x \right| < \frac{1}{2 q^2}
\end{equation*}
then $p/q$ is a convergent to $x$.
\end{prop}

\begin{prop}[\cite{Hardy}]\label{prop_p q minus q p}%Theorem 150 in \cite{Hardy}
$p_n q_{n-1} - p_{n-1}q_n = (-1)^{n+1}$.
\end{prop}

\begin{prop}[\cite{Hardy}]\label{prop_total quotient to number}%Page 141 of \cite{Hardy}
\begin{equation*}
    x = [a_0,a_1,\cdots,a_n,\alpha_{n+1}] = \frac{\alpha_{n+1} p_n + p_{n-1}}{\alpha_{n+1} q_n + q_{n-1}}
\end{equation*}
\end{prop}

%-----------------------------------------------------------------------------
\section{Reduced ideals and the convergents to $\omega_d$}\label{sec_Reduced ideals}

Following the literature of \cite{Dirichlet}, \cite{Ince} and \cite{Takagi}, an explicit correspondence between the set of reduced ideals of $\mathbb{Q}(\sqrt{d})$ and the set of reduced quadratic irrationals with discriminant $D$ was given in \cite{Yamamoto}. For the sake of references in following sections, a short material in \cite{Yamamoto} is included here.

Write $\omega = \frac{D + \sqrt{D}}{2}$ where $D$ is the discriminant of $\mathbb{Q}(\sqrt{d})$, so that the ring of integers is $O_d = \mathbb{Z}[\omega]$. For $x_1,\cdots,x_n \in \mathbb{Q}(\sqrt{d})$, let $[x_1,\cdots,x_n]$ and $(x_1,\cdots,x_n)$ be the $\mathbb{Z}$-module and $O_d$-module generated by $x_1,\cdots,x_n$ (for example $O_d = [1,\omega] = (1)$). For each integral ideal $I$ there is a unique canonical basis of the following form: $I=[a,b+c\omega]$ where $a,b,c$ are integers satisfying (i) $a>0$, $c>0$, $ac = N(I)$, (ii) $c$ divides $a,b$ and $N(I)$ divides $N(b+c\omega)$, and (iii) $-a < b + c \overline{\omega}<0$. We call the number
\begin{equation*}
\alpha(I) = \frac{b+c\omega}{a}
\end{equation*}
\emph{the quadratic irrational associated with the ideal $I$}. When $c=1$ and $\alpha(I)$ is reduced, we say that $I$ is reduced.

Two quadratic irrationals are said to be equivalent if their continued fraction expansions become identical at the tail, and the set of quadratic irrationals of discriminant $D$ falls into $h_d$ classes by this equivalence relation. By the above correspondence, the set of reduced ideals in $O_d$ gives the set of reduced quadratic irrationals of discriminant $D$, and a reduced ideal $I$ is in the principal ideal class if and only if $\alpha(I)$ is equivalent to $\omega_d$.

We begin with a well-known correspondence.
\begin{lem}[\cite{Cohen_3}, \S 5.6-5.7]\label{lem_XiAlpha}
\begin{equation*}
\alpha((\xi_n)) = \alpha_{n+1}
\end{equation*}
\end{lem}

\begin{lem}[\cite{Yamamoto}]\label{lem_ProdAlpha}
\begin{equation*}
\prod_{i = 1}^{l} \alpha_i = \varepsilon_d.
\end{equation*}
\end{lem}

\begin{lem}[\cite{Yamamoto}]\label{lem_ConditionReduced}
An integral ideal $I$ is reduced if (i) $N(I) < \sqrt{D}/2$ and (ii) the conjugate ideal $\overline{I}$ is relatively prime to $I$.
\end{lem}

\begin{lem}\label{lem_quotient_norm}
For $n \geq 0$
\begin{equation*}
\alpha_{n+1} =
\frac{\sqrt{D}}{\nu_n} - \frac{q_{n-1}}{q_n} + \frac{(-1)^{n+1}}{q_n\xi_n}. \quad \text{ In particular,}\quad \frac{\sqrt{D}}{\nu_n} - 1 < \alpha_{n+1} < \frac{\sqrt{D}}{\nu_n}.
\end{equation*}
\end{lem}

\begin{proof}
The cases $d \equiv$ 2 and 3 (mod 4) are easier in computation, so here we assume $d \equiv 1$ (mod 4) and $\omega_d = \frac{1 + \sqrt{d}}{2}$.
Recall that the continued fraction expansion of $\omega_d$ has a natural geometric interpretation on $xy$-plane. Let $O=(0,0)$ be the origin of the $xy$-plane, $A = (q_{n-1}, p_{n-1})$, $B = (q_{n}, p_{n})$, $C$ the intersection of $\overline{AB}$ and the line $y = \omega_d x$, and $D = (q_n, \omega_d q_n)$. Then $[\overline{AC}:\overline{CB}] = [\alpha_{n+1}:1]$ and the area of $\triangle{OAB}$ is 1/2. Observe that the area of $\triangle{OBD}$ is $\frac{1}{2}|(p_n - q_n \omega_d)q_n|$. Let $B' = (0,p_n)$, $D' = (0,\omega_d q_n)$.

We have
\begin{equation*}
\frac{\xi_n \overline{\xi_n}}{q_n^2} = \left( \frac{p_n}{q_n} - 1 + \omega_d \right) \left( \frac{p_n}{q_n} - 1 + 1 - \omega_d \right) = (-1)^{n+1} \frac{\nu_n}{q_n^2}
\end{equation*}
or
\begin{equation*}
\frac{p_n}{q_n} - \omega_d = \frac{(-1)^{n+1} \nu_n}{q_n (p_n - q_n + \omega_d q_n) }
\end{equation*}
and therefore
\begin{align*}
|\square B'B D D'| &= |(p_n - q_n \omega_d) q_n |\\
&= \frac{\nu_n}{p_n / q_n - 1 + \omega_d}\\
&= \frac{\nu_n}{2 \omega_d - 1 + (-1)^{n+1} \frac{\nu_n}{q_n (p_n - q_n + \omega_d q_n)}}\\
&= \frac{\nu_n}{2 \omega_d - 1} \left( \frac{1}{1 + (-1)^{n+1} \frac{\nu_n}{(2 \omega_d - 1) q_n (p_n - q_n + \omega_d q_n)} }  \right)\\
&= \frac{\nu_n}{\sqrt{d}} \left( 1 + \epsilon_n \right)^{-1}
\end{align*}

where $\epsilon_n = \frac{(-1)^{n+1} \nu_n}{\sqrt{d} q_n \xi_n}$. Examining the ratios of the coordinates of $A,B$ and $C$, it is easily deduced that the area of $\triangle{BCD}$ is $\frac{1 - q_{n-1}/q_n}{1 + \alpha_{n+1}} \frac{\nu_n}{2\sqrt{d}}\left( 1+ \epsilon_n  \right)^{-1} $, and hence
\begin{align*}
|\triangle{OBC}| & = |\triangle{OBD}| - |\triangle{BCD}| \\
&= \left(1 - \frac{1 - q_{n-1}/q_n}{1 + \alpha_{n+1}} \right) \frac{\nu_n}{2\sqrt{d}}\left( 1+ \epsilon_n  \right)^{-1}\\
&= \left(\frac{\alpha_{n+1} + q_{n-1}/q_n}{1 + \alpha_{n+1}} \right) \frac{\nu_n}{2\sqrt{d}}\left( 1+ \epsilon_n  \right)^{-1}
\end{align*}
But $|\triangle{OBC}| = |\triangle{OAB}| \frac{1}{1+\alpha_{n+1}} = \frac{1}{2(1+\alpha_{n+1})}$, whence $\left( \alpha_{n+1} + \frac{q_{n-1}}{q_n} \right) \frac{\nu_n}{\sqrt{d}}( 1+ \epsilon_n )^{-1} = 1$. Thus
\begin{equation*}
    \alpha_{n+1} = \frac{\sqrt{d}}{\nu_n} - \frac{q_{n-1}}{q_n} + \frac{(-1)^{n+1}}{q_n \xi_n}.
\end{equation*}
Recall that $q_{-1} = 0 < q_0 = 1 \leq q_1$, $q_{n-1} < q_n$ for $n \geq 2$, and $\xi_n \geq q_n \omega_d > q_n$ for $n \geq 0$. Therefore we have $\frac{\sqrt{D}}{\nu_n}-1 < \alpha_{n+1} < \frac{\sqrt{D}}{\nu_n}$ for $n \geq 0$, which proves the lemma in case $d \equiv 1$ (mod 4).

When $d \equiv$ 2 or 3 (mod 4), exactly the same computation with continued fraction of $\omega_d = \sqrt{d}$ completes the proof.
\end{proof}

Suppose $\xi \in \mathbb{Q}(\sqrt{d})$ is a quadratic integer with square-free norm $\mu$ and $(\mu,d)=1$. Then $(\xi)$ is a principal integral ideal which is relatively prime to its conjugate, and $(\xi^n)$ is relatively prime to its conjugate too. Hence the conditions of lemma~\ref{lem_ConditionReduced} are satisfied by $(\xi^m)$ if $|N(\xi^m)| \leq \omega_d - 1$. Combining lemma ~\ref{lem_XiAlpha}, ~\ref{lem_ProdAlpha} and ~\ref{lem_quotient_norm}, one easily sees that $\varepsilon_d$ is large if there are many quadratic integers $\xi \in \mathbb{Q}(\sqrt{d})$ with square-free norm $|N(\xi)| < \sqrt{D}/2$. In this sense, the problem of fundamental units is naturally translated to the problem of quadratic integers of small norms.

%-----------------------------------------------------------------------------
\section{Quadratic integers of small norms}\label{sec_Quadratic integers of small norms}

Assume $\xi = a + b \omega_d \in O_d$ where $a,b$ are positive integers that are relatively prime and $|N(\xi)| = \nu < \omega_d - 1$. Assuming $d \equiv 1$ mod 4 (or $d \equiv 2,3$ mod 4), it easily follows that $\frac{a+b}{b} - \omega_d < \frac{1}{2 b^2}$   (or $\frac{a}{b} - \omega_d < \frac{1}{2 b^2}$), which implies $\frac{a+b}{b}$ (or $\frac{a}{b}$) is a convergent to $\omega_d$ and hence $\xi$ comes from a convergent to $\omega_d$. Lemma~\ref{lem_ConditionReduced} in fact tells us that this is true if $\nu < \sqrt{D}/2$. Recall that for every positive integer $\nu$ there are only finitely many non-associated (quadratic) integers in $O_d$ of norm $\pm \nu$. As the unit rank of $O_d$ is 1, in each class of associated integers one can choose the least element among those irrational ones greater than 1. Let $F_{(d,\nu)} = \left\{ \xi_1, \cdots, \xi_t \right\}$ be the set of these least elements, and define $E_{\nu}(x) = \left| \mathbb{R}_{>1}^{<x} \bigcap \left(\bigcup_{d: squarefree} F_{(d,\nu)} \right) \right|$.

\begin{prop}\label{prop_minimal integers}
If $\omega_d = [a_0,\overline{a_1,\cdots,a_l}]$, $\omega_d -1 > \nu \geq 1$ and $\nu$ is square-free, then the elements of $F_{(d,\nu)}$ are of the form $a+b \omega_d$ where $\frac{a+b}{b}$ (or $\frac{a}{b}$) $= [a_0,a_1,\cdots,a_n]$ for some $n < l$.
\end{prop}

\begin{proof}
Clear from proposition~\ref{prop_Hardy}.
\end{proof}

Assume $p<\omega_d-1$ is a rational prime that splits or ramifies into principal prime ideals in $K_d/\mathbb{Q}$. Write $p O_d=P \overline{P}$. Let $\xi \in P$ and $\tilde{\xi} \in \overline{P}$ be the least elements of $P$, $\overline{P}$ among those greater than 1. Then $1 < \xi, \tilde{\xi} < \varepsilon_d$ and $\xi \tilde{\xi}$ is an algebraic integer associated to $p$. By proposition~\ref{prop_minimal integers} one can write $\xi = a + b \omega_d > \omega_d$ where $(a+b)/b$ or $a/b$ is a convergent to $\omega_d$, whence $|\overline{\xi}| = \frac{p}{\xi} < 1$. Thus $\tilde{\xi} \neq \overline{\xi}$ and $\xi \tilde{\xi}$ is not a rational integer. This shows that $\xi \tilde{\xi} = p\varepsilon_d$.

The distribution of minimal quadratic integers is given in the following
\begin{thm}\label{thm_E_mu_x}
Let $\nu < M < x$. Then
\begin{enumerate}
\item[(i)] $E_{\nu}(x) < 2 x - 2 \sqrt{\nu}+ O(1)$
\item[(ii)] $E_{\nu}(x) > 2\left(1-\frac{1}{2M-1}\right)x - \left( \sum_{\omega_d < M} \frac{|F_{(d,\nu)}|}{\log{\varepsilon_d}}\right) \log x + O(1)$
\end{enumerate}
\end{thm}

\begin{proof}
Observe that every quadratic integer $y$ of norm $\pm \nu$ is a solution of the equation $X^2 + mX \pm \nu = 0$ for some $m \in \mathbb{Z}$. The number of real quadratic integers $y$ greater than 1 with trace $m$ and norm $\pm \nu$ is 2 if $m > 2 \sqrt{\nu}$ and 1 if $m \leq 2 \sqrt{\nu}$. With the expression $\xi = \frac{m + \sqrt{m^2 \pm 4 \nu}}{2}$ the first inequality is trivial. As for the second inequality, note that such $y$ must be of the form $\xi \varepsilon_d^k$ for some $\xi \in F_{(d,\nu)}$ and $k \geq 0$. $E_{\nu}(x)$ counts the numbers with $k=0$, so we can simply exclude the numbers $\xi \varepsilon_d^k$ less than $x$ with $k \geq 1$. But $\xi \varepsilon_d^k < x$ if and only if $k < \frac{\log x - \log \xi}{\log {\varepsilon_d}}$, and therefore

\begin{equation*}
\sharp \left\{ \xi \varepsilon_d ^k < x \; \mid \; k \geq 1,\; \xi \in F_{(d,\nu)},\; \omega_d \leq M \right\} < \left( \sum_{\omega_d \leq M} \frac{|F_{(d,\nu)}|}{\log{\varepsilon_d}}  \right) \log x.
\end{equation*}

Now consider $\omega_d > M$ and write $\xi \varepsilon_d^k < x$ $\Longleftrightarrow$ $\xi < x$ and $\varepsilon_d^k < \frac{x}{\xi}$. Since $\omega_d > M > \nu$, as mentioned at the beginning of this section $\xi = n \acute{\xi}$ for some $n \in \mathbb{N}$ where $\acute{\xi}$ comes from a convergent to $\omega_d$ and hence $\xi \geq 2M -1$. Therefore the contribution to $E_{\nu}(x)$ from $\omega_d > M$ and $k \geq 1$ is less than the number of quadratic units in the interval $(1, \frac{x}{2M-1})$, which is $\frac{2}{2M-1} x + O(1)$.
\end{proof}

Let $a_{(d,\nu)}$ be the set of reduced integral ideals of norm $\nu$ in $O_d$. By lemma~\ref{lem_ConditionReduced}, if $\nu$ is square-free and $\nu < \sqrt{D}/2$ then $|F_{(d,\nu)}| \leq |a_{(d,\nu)}|$.

\begin{prop}\label{prop_NumberofIdeals}
Assume $\omega_d > \nu$ where $\nu$ is square-free. Let $\nu_1 = gcd(\nu, 2d)$ and write $\nu = \nu_1 \nu_2$. Then
\begin{equation*} |a_{(d,\nu)}| = \begin{cases}
2^{\omega(\nu_2)} & \text{if $d$ is a square modulo $\nu$, $d \equiv$ 2 or 3 (mod 4)}\\
2^{\omega(\nu_2)} & \text{if $d$ is a square modulo $\nu$, $d \equiv$ 1 (mod 4), $\nu$ is odd}\\
2^{\omega(\nu_2)+1} & \text{if $d$ is a square modulo $\nu$, $d \equiv$ 1 (mod 8), $\nu$ is even  }\\
0 & \text{otherwise}
 \end{cases}
\end{equation*}
\end{prop}

\begin{proof}
Let $I \in a_{(d,\nu)}$ and write $\alpha(I) = \frac{b + c \omega}{\nu}$. Then $I$ is reduced if and only if $c = 1$ and $\alpha(I)>1$, $-1 < \overline{\alpha(I)} < 0$.

Suppose $d \equiv$ 2 or 3 (mod 4) so that $\alpha(I) = \frac{b + 2d + \sqrt{d}}{\nu}$. The condition $N(b + 2d + \sqrt{d}) \equiv 0$ (mod $\nu$) implies $(b + 2d)^2 \equiv d$ (mod $\nu$), so we can write $b \equiv -2 d + \zeta$ (mod $\nu$) where $\zeta^2 \equiv d$ (mod $\nu$). Hence if $d$ is not a square modulo $\nu$ there is no such ideal. The condition $-\nu < b + c \overline{\omega} < 0$ implies $b$ varies in a complete system of residues modulo $\nu$; hence if $d$ is a square modulo $\nu$, the number of possible $b$'s is the number of solutions to $\zeta^2 \equiv d$ (mod $\nu$). For an odd prime factor $p$ of $\nu$, the congruence $\zeta^2 \equiv d$ (mod $p$) has two roots if $p \nmid d$ and one root if $p \mid d$. The first case easily follows from this.

Now suppose $d \equiv$ 1 (mod 4). Then  $\alpha(I) = \frac{b + d/2 + \sqrt{d}/2}{\nu}$ and in the same way as above we get $b^2 + d b + d(d-1)/4 \equiv$ 0 (mod $\nu$) where $b$ varies in a complete system of residues modulo $\nu$. When $\nu$ is odd, 2 is a unit modulo $\nu$ so one can write $b \equiv - \frac{d}{2} + \zeta$ (mod $\nu$) where $\zeta^2 \equiv \frac{d}{4}$ (mod $\nu$). This proves the second case. When $\nu$ is even, write $\nu = 2 \nu'$ and consider $b^2 + d b + d(d-1)/4 \equiv$ 0 (mod 2) and $b^2 + d b + d(d-1)/4 \equiv$ 0 (mod $\nu'$) separately. The latter has $2^{\omega(\nu_2)}$ solutions. The former has no solution when $d \equiv 5$ (mod 8) and two solutions when $d \equiv 1$ (mod 8), which proves the third case.
\end{proof}

%-----------------------------------------------------------------------------
\section{Proof of Theorem~\ref{thm_main}}\label{sec_proof of main theorem}

Now Theorem~\ref{thm_main} can be proved easily in the philosophy of Theorem 3.1 in \cite{Yamamoto}.

\begin{proof}[proof of theorem~\ref{thm_main}]
Let $\xi \in F_{(d,|\mu|)}$ and put $L = \lfloor \log_{|\mu|} (\sqrt{D}/2) \rfloor$. The ideals $(\xi^e)$ and $(\overline{\xi}^e)$ of $O_d$ are reduced if $|\mu|^e < \sqrt{D}/2$ by lemma~\ref{lem_ConditionReduced}. Now by lemma~\ref{lem_XiAlpha}, ~\ref{lem_ProdAlpha} and ~\ref{lem_quotient_norm} we have

\begin{align*}
\varepsilon_d &= \prod_{i=1}^{l} \alpha_i\\
&\geq \prod_{e=1}^{L} \alpha((\xi^e)) \prod_{e=1}^{L} \alpha((\overline{\xi}^e))\\
&> \prod_{e=1}^{L} \left( \frac{\sqrt{D}}{|\mu|^e} - 1 \right)^2.
\end{align*}

Taking logarithm,

\begin{align*}
\log \varepsilon_d &> \sum_{e=1}^{L} 2 \log\left(\frac{\sqrt{D}}{|\mu|^e} - 1\right)\\
&> 2 \sum_{e=1}^{L} \left(  \log \sqrt{D} -  e \log |\mu| - \left( \frac{\sqrt{D}}{|\mu|^e}-1\right)^{-1} \right).
\end{align*}

The last term can be written
\begin{align*}
\sum_{e=1}^L \left( \frac{\sqrt{D}}{|\mu|^e}-1\right)^{-1} &= \sum_{e=1}^L \frac{|\mu|^e}{\sqrt{D}} \left(  \frac{1}{1 - \frac{|\mu|^e}{\sqrt{D}}}  \right)\\
&< \sum_{e=1}^L \frac{2|\mu|^e}{\sqrt{D}}\\
&= \frac{2|\mu|^e}{\sqrt{D}} \left( \frac{|\mu|^L - 1}{|\mu|-1}\right)\\
&< \frac{|\mu|}{|\mu|-1},
\end{align*}

and therefore
\begin{align*}
\log \varepsilon_d &> 2 L \log \sqrt{D} - L (L + 1) \log |\mu| - \frac{2|\mu|}{|\mu|-1}\\
&> 2 \left(  \frac{\log(\sqrt{D}/2)}{\log |\mu|} - 1  \right) \log \sqrt{D} - \log (\sqrt{D}/2) \left(  \frac{\log(\sqrt{D}/2)}{\log |\mu|} + 1  \right) - \frac{2|\mu|}{|\mu|-1}\\
&= \frac{1}{\log |\mu|} \left( \log{ \frac{\sqrt{D}}{2}  } \right)^2 - \left( 3 - \frac{2 \log 2}{\log |\mu|} \right) \log{\frac{\sqrt{D}}{2}} - -2\log 2 -\frac{2|\mu|}{|\mu|-1}.
\end{align*}

\eqref{eqn_thm_main_1} follows from theorem~\ref{thm_E_mu_x} and proposition~\ref{prop_NumberofIdeals} immediately. \eqref{eqn_thm_main_2} also follows from Theorem~\ref{thm_E_mu_x} at once too.
\end{proof}

% -----------------------------------------------------------
\section{The quadratic progressions}\label{sec_the quadratic progressions}

In this section we give the proofs of proposition~\ref{prop_d_mu} and ~\ref{prop_d_squarefree}. The constructive proofs also give quadratic progressions, which resemble the progressions that appeared in the inverse problem for Pell equation \cite{JHPark2}.

\medskip
\begin{proof}[Proof of proposition~\ref{prop_d_mu}]
We prove (1) and (2) at the same time. Suppose $d\in\frak{D}$. Then there exists a quadratic integer $\xi \in O_d$ of norm $\mu$. Assume $\xi \in \mathbb{Z}[\sqrt{d}]$ first and write $\xi = ny+x+y\sqrt{d}$. Multiplying a unit if necessary, we may assume $0\leq x < y$ and $1 < \xi < \varepsilon_d$. Then
\begin{equation*}
    N(\xi) = n^2 y^2 + 2nxy + x^2 - y^2 d = \mu
\end{equation*}
and
\begin{equation*}
    x^2 \equiv \mu \text{ (mod $y$)} \;\;\;\text{and} \;\;\; 2xyn \equiv - x^2 + \mu \text{ (mod $y^2$)}.
\end{equation*}
Since $\mu$ is square-free, $(x,y)=1$ and so $2xyn \equiv - x^2 + \mu$ (mod $y^2$) if and only if $2 n \equiv  \frac{-x^3 + \mu x}{y} \mu^{-1}$ (mod $y$). Thus $(y,x)\in I^{(0)}(\mu)$, and $n$ is uniquely determined modulo $\tilde{y}$. We write $n = n_0 + \tilde{y} k$ for this.

Conversely, if $x^2 \equiv \mu$ (mod $y$) and $2xyn \equiv - x^2 + \mu$ (mod $y^2$), put
\begin{equation}\label{eqn_d}
    d = n^2 + \frac{2x}{y}n + \frac{x^2 - \mu}{y^2}
\end{equation}
and it immediately follows that
\begin{equation*}
    N(ny + x + y \sqrt{d}) = \mu.
\end{equation*}
We have proved that for each $d \in \frak{D}$, in case $\xi \in \mathbb{Z}[\sqrt{d}]$, there exists a pair $(y,x) \in I^{(0)}(\mu)$, and for each $(y,x)\in I^{(0)}(\mu)$ there arises an arithmetic progression $\{ n_0 + \tilde{y} k\}_k$ with common difference $\tilde{y}$, which gives rise to a quadratic progression $Q^{(0)} = \{ d(k) = (n_0+\tilde{y} k)^2 + \frac{2x}{y}(n_0+\tilde{y} k) + \frac{x^2-\mu}{y^2} \}_k$. It is easy to see that $Q^{(0)}$ is of the same form with $\frak{D}^{(0)}(\mu;y;t)$ for some $t$.

Now assume $\xi = n y + x + y \omega_d \in \mathbb{Z}[\omega_d] \setminus \mathbb{Z}[\sqrt{d}]$. Then

\begin{equation*}
    N(ny+x+y\omega_d) = \left(ny+x+\frac{y}{2}\right)^2 - \frac{y^2}{4}d = \mu
\end{equation*}
or
\begin{equation*}
    ((2n+1)y+2x)^2 - y^2d = 4\mu.
\end{equation*}
Since $y$ is odd, we have
\begin{equation*}
    x^2 \equiv \mu \text{ (mod $y$)} \;\;\;\text{and} \;\;\; (2n+1)xy \equiv - x^2 + \mu \text{ (mod $y^2$)}
\end{equation*}
and $n$ is uniquely determined modulo $y = \tilde{y}$.

Conversely, if $x^2 \equiv \mu$ (mod $y$) and $(2n+1)xy \equiv -x^2+\mu$ (mod $y^2$), put
\begin{equation}\label{eqn_d 1 mod 4 EQN}
    d = (2n+1)^2 + \frac{4x}{y}(2n+1) + \frac{4x^2 - 4\mu}{y^2}
\end{equation}
so that $N(ny + x + y \omega_d) = \mu$. This gives another quadratic progression $Q^{(1)}$ which is of the same form as $\frak{D}^{(1)}(\mu;y;t)$ for some $t$.

There is a canonical way of choosing $\phi_{\mu}^{(j)}$. Let $n$ and $d$ be in the relation as in (\ref{eqn_d}). For each $(y,x)\in I^{(j)}(\mu)$, $Q^{(j)}$ have only finitely many $d$ for which $n y + x + y \omega_d > \varepsilon_d$. This is because $|\mu| < \frac{\sqrt{D}}{2}$ implies $n+\frac{x}{y}$ (or $n + 1 +\frac{x}{y}$) is a convergent to $\omega_d$; in this case, lemma~\ref{lem_quotient_norm} shows that $n y + x + y \omega_d > \varepsilon_d$ can happen only when $\sqrt{D} = a + O(1)$ where $a$ is the largest partial quotient in $\frac{x}{y}=[0,a_1,a_2,\cdots,a_m]$. Discarding these finite numbers, we choose $\phi_{\mu}^{(j)}(y,x)$ to be the smallest $d$ in $Q^{(j)}$ satisfying $n y + x + y \omega_d < \varepsilon_d$.

In previous paragraphs, for each $d \in \frak{D}^{(j)}(\mu;y;\phi_{\mu}^{(j)}(y,x))$ there is a quadratic integer $ny + x + y \omega_d$ whose trace is $T^{(0)}_{y,x}(d) = 2ny + 2x$ or $T^{(1)}_{y,x}(d) = (2n+1)y + 2x$. It is clear that $d_{\mu}(T^{(j)}_{y,x}(d)) = d$ for $d \in \frak{D} \cap \frak{D}^{(j)}(\mu;y;\phi_{\mu}^{(j)}(y,x))$ and the traces $T^{(j)}_{y,x}(d)$ form an arithmetic progression from the expression $n = n_0 + \tilde{y}k$. To compute the sum \eqref{eqn_d_T_zeta sum}, observe that $n \asymp \sqrt{d}$ and write
\begin{align*}
    &\sum_{j=0}^1 \sum_{(y,x)\in I^{(j)}(\mu)} \sum_{d \in \frak{D}^{(j)}(\mu;y;\phi_{\mu}^{(j)}(y,x))} \frac{1}{T^{(j)}_{y,x}(d)^s}\\
    &\ll \sum_{j=0}^1 \sum_{(y,x)\in I^{(j)}(\mu)} \sum_{ k \geq 0} \frac{1}{\left(x + y \sqrt{\phi^{(j)}_{\mu}(y,x)} + y^2k\right)^s}\\
    &\ll \zeta(s) \left( \sum_{j=0}^1 \sum_{(y,x)\in I^{(j)}(\mu)} \frac{1}{y^{2s}} \right)
    + \sum_{j=0}^1 \sum_{(y,x)\in I^{(j)}(\mu)}  \frac{1}{\left(y \sqrt{\phi^{(j)}_{\mu}(y,x)}\right)^s}
    \qquad  \text{as $s \rightarrow 1+$.}
\end{align*}
For $(y,x) \in I^{(j)}(\mu)$, every odd prime factor $q$ of $y$ satisfies $\left( \frac{\mu}{q} \right) = 1$. Unless $\mu = 1$, the sum $\sum \frac{1}{y^s}$ over all such $y$'s involves only a half of the primes in the Euler product form of the zeta function. Hence its order is asymptotically $\asymp \zeta(s)^{1/2}$ as $s \rightarrow 1+$. Using the Chinese remainder theorem it is easy to see that the number of $x \in [0,y)$ satisfying $x^2 \equiv \mu$ (mod $y$) is a bounded multiple (that is, between a half and twice) of $2^{\omega(y)}$. Recalling $\sum_{n=1}^{\infty} \frac{2^{\omega(n)}}{n^s}= \frac{\zeta(s)^2}{\zeta(2s)}$ (see theorem 301 of \cite{Hardy}), one sees that
\begin{equation*}
    \sum_{(y,x) \in I^{0}(\mu)} \frac{1}{y^s} \asymp \sum_{y\; :\;\text{ $\mu$ is a square mod $y$}} \frac{2^{\omega(y)}}{y^s} \asymp \zeta(s) \;\text{ as $s \rightarrow 1+$}
\end{equation*}
and $\sum_{j=0}^1 \sum_{(y,x)\in I^{(j)}(\mu)} y^{-2s} \ll 1$. Besides, since there is an upper bound of the period of $\omega_d$ in terms of $d$ \cite{Cohn,Ishii} which gives an upper bound of $\varepsilon_d$, we also have a lower bound of $\phi_{\mu}^{(j)}(y,x)$ (in terms of $y,x$) which goes to the infinity as $y$ grows. Therefore
\begin{equation*}
    \sum_{(y,x) \in I^{(j)}(\mu)} \frac{1}{\left(y \sqrt{\phi_{\mu}^{(j)}(y,x)}\right)^s} = \;_o(\zeta(s)) \quad \text{ as $s \rightarrow 1+$},
\end{equation*}
which completes the proof.

\end{proof}

\medskip
Now we give the number of elements in $\frak{D}(\mu;y,x)$ that are square-free. We state a lemma first.

\begin{lem}\label{lem_Quadratic Hensel lemma}
Let $p$ be an odd prime and $f(x) \in \mathbb{Z}[x]$ a quadratic polynomial
whose leading coefficient is not divisible by $p$. Then $f(x) \equiv 0$ (mod $p^m$) has a solution if and only if the discriminant of $f(x)$ is a square modulo $p^m$.
\end{lem}

\begin{proof}
($\Leftarrow$) The root formula for quadratic equations gives a solution.

($\Rightarrow$) Let $t \in \mathbb{Z}$ be a solution to the modular equation. Then
\begin{equation*}
    f(t + p j) = f(t) + f'(t) p j + \frac{f''(t)}{2} p^2 j^2.
\end{equation*}

Let $p^r|| f'(t)$. Choose $\alpha$ so that $f(t) + \alpha \equiv 0$ (mod $p^M$) where $M$ is sufficiently large. Note that $\alpha$ is necessarily divisible by $p^m$. We have
\begin{align*}
    f(t + p^{r+1}j_1) + \alpha &= f(t) + \alpha + \frac{f'(t)}{p^r}p^{2r+1}j_1 + \frac{f''(t)}{2}p^{2r+2}j_1^2 \\
&\equiv f(t) + \alpha + \frac{f'(t)}{p^r}p^{2r+1}j_1 \text{ (mod $p^{2r+2}$)},
\end{align*}
\begin{align*}
    f(t+p^{r+1}j_1+p^{r+2}j_2) + \alpha &= f(t+p^{r+1}j_1) + \alpha + f'(t+p^{r+1}j_1) p^{r+2}j_2\\
&\;\;\;\; + \frac{f''(t+p^{r+1}j_1)}{2}p^{2r+4}j_2^2. \\
\end{align*}
Writing $f'(t+p^{r+1}j_1) = f'(t)+f''(t)p^{r+1}j_1$,
\begin{align*}
    f(t+p^{r+1}j_1+p^{r+2}j_2) + \alpha &\equiv f(t+p^{r+1}j_1) + \alpha + \frac{f'(t)}{p^r}p^{2r+2}j_2 \text{ (mod $p^{2r+3}$)},
\end{align*}
and successively, there exists a unique sequence $(j_1,j_2,j_3,\cdots)$ such that $t+p^{r+1}j_1 + p^{r+2}j_2+p^{r+3}j_3+\cdots \in \mathbb{Z}_p$ is a root of $f(x)+\alpha$. Since $\mathbb{Z}_p [x]$ is a UFD, it follows that the discriminant of $f(x)+\alpha$ is a square in $\mathbb{Z}_p$ and hence that of $f(x)$ is a square modulo $p^m$.
\end{proof}

We appeal to the next theorem. We only need its strength for quadratic polynomials, which can be proved unconditionally using the sieve of Eratosthenes \cite{Granville}.

\begin{thm}[\cite{Granville}]\label{thm_squarefree portion}
Suppose that $f(x) \in \mathbb{Z}[x]$ has no repeated root. Let $B$ be the largest integer which divides $f(n)$ for all integer $n$, and select $B'$ to be the smallest divisor of $B$ for which $B/B'$ is square-free. If the $abc$-conjecture is true, then there are $\sim C_f N$ positive integers $n \leq N$ for which $f(n)/B'$ is square-free, where $C_f>0$ is a positive constant, which we determine as follows;
\begin{equation*}
C_f = \prod_{p:\; \text{prime}} \left(  1 - \frac{\omega_f(p)}{p^{2 + q_p}}  \right)
\end{equation*}
where, for each prime $p$, we let $q_p$ be the largest power of $p$ which divides $B'$ and let $\omega_f(p)$ denote the number of integers $a$ in the range $1 \leq a \leq p^{2+q_p}$ for which $f(a)/B' \equiv 0$ (mod $p^2$).
\end{thm}

%%%%%%%%%%%%%%%%%%%%%%%%%%%%%%%%%%%%%%%%%%%%%%%%%%%%%%%%%%%%%%%%%%%%%%%%%%%%%%%%%%%
\begin{proof}[Proof of proposition~\ref{prop_d_squarefree}]
We first treat $\frak{D}^{(0)}(\mu;y,\phi_{\mu}^{(0)}(y,x))$. Let $d_0 = \phi_{\mu}^{(0)}(y,x)$ be the least element of $\frak{D}^{(0)}(\mu;y,\phi_{\mu}^{(0)}(y,x))$ and $n_0$ the associated integer in the context of the proof of proposition~\ref{prop_d_mu}. The elements of $\frak{D}^{(0)}(\mu;y,\phi_{\mu}^{(0)}(y,x))$ are given by a quadratic polynomial
\begin{align*}\label{eqn_d_mu_polynomial form}
    d = d(k) &=\left( n_0 + k \tilde{y} \right)^2 + 2 \frac{x}{y}\left( n_0 + k \tilde{y} \right) + \frac{x^2 - \mu}{y^2}\\
&= \tilde{y}^2 k^2 + \left( \frac{2\tilde{y}}{y}x + 2 \tilde{y} n_0 \right) k + d_0
\end{align*}
for nonnegative integers $k$.

We use the notations of theorem~\ref{thm_squarefree portion}. Let $\delta$ be the discriminant of the quadratic polynomial $d(k)$ and $\omega'_d(p)$ the number of solutions to $d(k) \equiv 0$ (mod $p^2$) in the range $0 \leq k < p^2$. When $y$ is even, $\delta = (x + y n_0)^2 - y^2 d_0 = \mu$ and similarly $\delta = 4\mu$ when $y$ is odd. Therefore $d(k)$ has no repeated root. Note that $\omega'_d(p) < p^2$ implies $p \nmid B'$ and hence $q_p=0$.

We consider odd primes first. Clearly $d(k)$ (mod $p$) is degenerate if and only if $p | \tilde{y} \mu$.
For $p \nmid \tilde{y} \mu$, every root of $d(k) \equiv 0$ (mod $p$) has a unique lifting to a $p$-adic root. There exists such a root if and only if the discriminant is a square modulo $p$, whence we have
\begin{equation*}
    w'_d(p) =\left\{
               \begin{array}{ll}
                 2 & \hbox{if $\left( \frac{\mu}{p} \right) = 1$;} \\
                 0 & \hbox{if $\left( \frac{\mu}{p} \right) = -1$.}
               \end{array}
             \right.
\end{equation*}

If $p | \tilde{y}$, $d(k)$ is congruent to $(x + y n_0)k + d_0$ or $2(x + y n_0)k + d_0$ modulo $p^2$, which has a unique solution and hence $\omega'_d(p) = 1$.

When $p | \mu$, $d(k) \equiv 0$ (mod $p$) has a double root. Let $t$ be the root of this equation in the range $0 \leq t < p$. From $d(t+pj) \equiv d(t) + d'(t)pj$ (mod $p^2$) and $d'(t) \equiv 0$ (mod $p$), it follows that $d(k) \equiv 0$ (mod $p^2$) has $p$ roots if $d(t) \equiv 0$ (mod $p^2$) or none otherwise. By lemma~\ref{lem_Quadratic Hensel lemma}, $d(k) \equiv 0$ (mod $p^2$) has a root if and only if the discriminant $\mu$ (or $4\mu$) is a square modulo $p^2$, which in this case is equivalent to $p^2 | \mu$. Thus $\omega'_d(p) = p$ if $p^2 | \mu$ and $\omega'_d(p) = 0$ if not.

Now let $p=2$.
Assume $y$ is odd(so $2y \equiv 2$ mod $4$). Then
\begin{align*}
    d(k) &\equiv k^2 + 2(x + y n_0)k + d_0\\
&\equiv k^2 + 2(x+n_0)k + n_0^2 + 2 x n_0 + x^2 - \mu \text{ (mod $4$)}
\end{align*}
and
\begin{align*}
&d(0) \equiv d(2) \equiv n_0^2 + 2 x n_0 + x^2 -\mu \text{ (mod $4$)},\\
&d(1) \equiv d(3) \equiv n_0^2 + 2 x n_0 + x^2 -\mu + 1 + 2 x + 2 n_0 \text{ (mod $4$)}.
\end{align*}
If $n_0$ is odd, $d(0) \equiv d(2) \equiv (x+1)^2 -\mu$ (mod $4$) and $d(1) \equiv d(3) \equiv x^2 - \mu$ (mod $4$). If $n_0$ is even, $d(0) \equiv d(2) \equiv x^2 -\mu$ (mod $4$) and $d(1) \equiv d(3) \equiv (x+1)^2 - \mu$ (mod $4$). It follows that $\omega'_d(2) = 2$ if $\mu \equiv 0,1$ (mod $4$) and $\omega'_d(2) = 0$ otherwise.

Now assume $y$ is even (and $x$ is necessarily odd). Suppose $y = 2 \tilde{y}$ where $\tilde{y}$ is odd. In a single line of computation we obtain
\begin{align*}
&d(1) \equiv d_0 + x + 3 \text{ (mod $4$)},\\
&d(2) \equiv d_0 + 2x \text{ (mod $4$)},\\
&d(3) \equiv d_0 + 3x + 3 \text{ (mod $4$)}
\end{align*}
which shows that $\omega'_d(2) = 2$ when $d_0$ is even and $\omega'_d(2) = 0$ otherwise. Observe that
\begin{equation*}
\mu = (x + y n_0)^2 - y^2 d_0 \equiv 1 + 4 n_0 + 4 n_0^2 - 4 d_0 \equiv 1 - 4 d_0 \text{ (mod $8$)}
\end{equation*}
and it follows that $d_0$ is even if and only if $\mu \equiv 1$ (mod $8$).

Finally, suppose $4 | y$. In this case $d(k) \equiv x k + d_0$ (mod $4$), which has a unique solution to $d(k) \equiv 0$ (mod $4$) and hence $\omega'_d(2) = 1$.

In every case $\omega'_d(p)$ is less than $p^2$ and $q_p = 0$.

To treat $\frak{D}^{(1)}(\mu;y,\phi_{\mu}^{(1)}(y,x))$, write
\begin{align*}
    d =d(k) &= (2(n_0+yk)+1)^2 + \frac{4x}{y}(2(n_0+yk)+1) + \frac{4x^2 - 4\mu}{y^2}\\
    &= 4 \left( n_0^2 + 2 y n_0 k + y^2 k^2 \right) + 1 + 4 n_0 + 4 y k + \frac{4x}{y}(2n_0 + 1 + 2yk)\\
    &\;\;\;\;\; + \frac{4x^2 - 4\mu}{y^2}\\
    &= 4 y^2 k^2 + 4(2 y n_0 + y + 2x)k + d_0.
\end{align*}
The computations for odd primes are exactly the same as the case $j=0$. For $p=2$, $\omega'_d(2)=0$ because $d$ is always congruent to 1 modulo 4. Applying theorem~\ref{thm_squarefree portion} we complete the proof.
\end{proof}

% -----------------------------------------------------------
\section{The Density of Discriminants and further topics }\label{sec_further topics}

It is very natural to ask the density of $d$'s for which $p$ splits or ramifies into principal prime ideals, or $d$'s such that $p$ is in the image of the norm map $N_{\mathbb{Q}(\sqrt{d}) / \mathbb{Q}}$. This is trivially zero because such $d$ must not be divisible by any prime factor $q$ with $\left( \frac{p}{q} \right) = -1$ and these `special' integers constitute only a null set in $\mathbb{Z}$. More meaningful question is therefore to ask the portion of $d$'s out of all those special integers.

When $\mu \neq 1$, the constructions of $I^{(j)}(\mu)$ and $\widehat{\frak{D}}^{(j)}(\mu;y;\phi_{\mu}^{(j)}(y,x))$ are pretty much the same as the case $\mu = -1$. We can therefore expect the density of
\begin{equation*}
\bigcup_{(y,x) \in I^{(j)}(\mu)} \widehat{\frak{D}}^{(j)}(\mu;y;\phi_{\mu}^{(j)}(y,x))
\end{equation*}
for $\mu \neq 1$ to be always similar to that of
\begin{equation*}
\bigcup_{(y,x) \in I^{(j)}(-1)} \widehat{\frak{D}}^{(j)}(-1;y,\phi_{-1}^{(j)}(y,x)).
\end{equation*}

This counts the square-free integers $d$ such that $O_d$ has an element of norm $-1$, i.e. the fundamental unit of $\mathbb{Q}(\sqrt{d})$ has norm $-1$. This is possible only when every odd prime factor of $d$ is congruent to $1$ modulo $4$, but not all of such integer $d$ gives a field with $N(\varepsilon_d) = -1$. Assume $\frak{P}$ is a set of prime numbers with Dirichlet density $\sigma$. Following the estimation in \cite{Rieger1965}, one can deduce that the number of positive integers (or positive square-free integers) less than $N$ and whose odd prime divisors are all in $\frak{P}$ is of the order $\asymp N \left( \log N \right) ^{-1+\sigma}$. In particular, when $\mu$ is not a square and $\frak{P} = \{p\;|\;p:\text{prime, } \left(\frac{\mu}{p}\right)=1\}$, the number of such fundamental discriminants is of the order $\frac{N}{\sqrt{\log N}}$.

For $\mu = -1$, it is a recent result that between $41\%$ and $67\%$ out of such fundamental discriminants satisfies $N(\varepsilon_d) = -1$ \cite{Fouvry_3}. We hope similar results to be found for prime numbers $p$ instead of $-1$ too, but the situation is not that simple. Consider a prime ideal $\frak{p}$ of $\mathbb{Q}(\sqrt{d})$ above $p$ and its ideal class $[\frak{p}]$. Then our problem is to show that $[\frak{p}]$ is the principal class for a positive density of fundamental discriminants out of those $O(N (\log N)^{-1/2})$ numbers. As $-1$ is replaced by $p$, however, the argument in \cite{Fouvry_3} only implies that the order of $[\frak{p}]$ is not divisible by 2. This is because the whole reasoning stems from the theory of genera, which covers the 2-torsion elements (and 2-divisibility) in the ideal class group. The asymptotic behavior of 3-torsion elements is handled via class field theory \cite{Davenport}, and the same technique seems to be applicable in obtaining some 3-divisibility result of $[\frak{p}]$. Except these few results, not so much is known about density estimation. It will be very interesting if a family of ideal classes can be actually shown to be principal.

% -----------------------------------------------------------
\bibliographystyle{amsplain}
\bibliography{NQI2013}
% -----------------------------------------------------------

\end{document}